\documentclass[12pt]{article}%
\usepackage{amsmath}
\usepackage{amsfonts}
\usepackage{amssymb}
\usepackage{graphicx}%
\providecommand{\U}[1]{\protect\rule{.1in}{.1in}}
\newtheorem{theorem}{Theorem}

\newtheorem{corollary}[theorem]{Corollary}

\newtheorem{proposition}[theorem]{Proposition}

\newenvironment{proof}[1][Proof]{\noindent\textbf{#1.} }{\ \rule{0.5em}{0.5em}}
\begin{document}

\title{Laguerre-Freud equations for Generalized Hahn polynomials of type I}
\author{Diego Dominici \thanks{e-mail: dominicd@newpaltz.edu}\\Department of Mathematics \\State University of New York at New Paltz \\1 Hawk Dr.\\New Paltz, NY 12561-2443 \\USA}
\maketitle

\begin{abstract}
We derive a system of difference equations satisfied by the three-term
recurrence coefficients of some families of discrete orthogonal polynomials.

\end{abstract}

\section{Introduction}

Let $\left\{  \mu_{n}\right\}  $ be a sequence of complex numbers and
$L:\mathbb{C}\left[  x\right]  \rightarrow\mathbb{C}$ be a linear functional
defined by%
\[
L\left[  x^{n}\right]  =\mu_{n},\quad n=0,1,\ldots.
\]
Then, $L$ is called the \emph{moment functional} determined by the formal
moment sequence $\left\{  \mu_{n}\right\}  $. The number $\mu_{n}$ is called
the \emph{moment} of order $n$. A sequence $\left\{  P_{n}\left(  x\right)
\right\}  \subset\mathbb{C}\left[  x\right]  ,$ of monic polynomials with
$\deg\left(  P_{n}\right)  =n$ is called an \emph{orthogonal polynomial
sequence} with respect to\ $L$ provided that \cite{MR0481884}%
\[
L\left[  P_{n}P_{m}\right]  =h_{n}\delta_{n,m},\quad n,m=0,1,\ldots,
\]
where $h_{n}\neq0$ and $\delta_{n,m}$ is Kronecker's delta.

Since
\[
L\left[  xP_{n}P_{k}\right]  =0,\quad k\notin\left\{  n-1,n,n+1\right\}  ,
\]
the monic orthogonal polynomials $P_{n}\left(  x\right)  $ satisfy the
\emph{three-term recurrence relation}
\begin{equation}
xP_{n}\left(  x\right)  =P_{n+1}\left(  x\right)  +\beta_{n}P_{n}\left(
x\right)  +\gamma_{n}P_{n-1}\left(  x\right)  , \label{3-term}%
\end{equation}
where%
\begin{equation}
\beta_{n}=\frac{1}{h_{n}}L\left[  xP_{n}^{2}\right]  ,\quad\gamma_{n}=\frac
{1}{h_{n-1}}L\left[  xP_{n}P_{n-1}\right]  . \label{betag}%
\end{equation}
If we define $P_{-1}\left(  x\right)  =0,$ $P_{0}\left(  x\right)  =1,$ we see
that
\begin{equation}
P_{1}\left(  x\right)  =x-\beta_{0}, \label{P1}%
\end{equation}
and%
\begin{equation}
P_{2}\left(  x\right)  =\left(  x-\beta_{1}\right)  \left(  x-\beta
_{0}\right)  -\gamma_{1}. \label{P2}%
\end{equation}
Because%
\[
L\left[  xP_{n}P_{n-1}\right]  =L\left[  P_{n}^{2}\right]  ,
\]
we have%
\begin{equation}
\gamma_{n}=\frac{h_{n}}{h_{n-1}},\quad n=1,2,\ldots, \label{gh}%
\end{equation}
and we define
\begin{equation}
\gamma_{0}=0. \label{g0}%
\end{equation}
Note that from (\ref{betag}) we get%
\begin{equation}
\beta_{0}=\frac{1}{h_{0}}L\left[  x\right]  =\frac{\mu_{1}}{\mu_{0}}.
\label{beta00}%
\end{equation}

If the coefficients $\beta_{n},\gamma_{n}$ are known, the recurrence
(\ref{3-term}) can be used to compute the polynomials $P_{n}\left(  x\right)
.$ Stability problems and numerical aspects arising in the calculations have
been studied by many authors \cite{MR0213062}, \cite{MR2383071},
\cite{MR0221789}, \cite{MR727118}.

If explicit representations of the polynomials $P_{n}\left(  x\right)  $ are
given, symbolic computation techniques can be applied to obtain recurrence
relations and, in particular, to find expressions for the coefficients
$\beta_{n},\gamma_{n}$ (see \cite{MR1766263}, \cite{MR2768529},
\cite{MR1379802}, \cite{MR0107725}, \cite{MR1090884}).

If, alas, the only knowledge we have is the linear functional $L,$ the
computation of $\beta_{n}$ and $\gamma_{n}$ is a real challenge. One
possibility is to use the Modified Chebyshev algorithm \cite[2.1.7]%
{MR2061539}. Another is to obtain recurrences for $\beta_{n},\gamma_{n}$ of
the form \cite{MR1773440}, \cite{MR2451211}%
\begin{align*}
\gamma_{n+1}  &  =F_{1}\left(  n,\gamma_{n},\gamma_{n-1},\ldots,\beta
_{n},\beta_{n-1},\ldots\right)  ,\\
\beta_{n+1}  &  =F_{2}\left(  n,\gamma_{n+1},\gamma_{n},\ldots,\beta_{n}%
,\beta_{n-1},\ldots\right)  ,
\end{align*}
for some functions $F_{1},F_{2}.$ This system of recurrences is known as the
\emph{Laguerre-Freud equations} \cite{MR0419895}, \cite{MR3363411}. The name
was coined by Alphonse Magnus as part of his work on Freud's conjecture
\cite{MR839005}, \cite{MR835728}, \cite{MR1340938}, \cite{MR1705232}. In terms
of performance, the Modified Chebyshev algorithm requires $O\left(
n^{2}\right)  $ operations, while the Laguerre-Freud equations require only
$O\left(  n\right)  $ operations for the computation of $\beta_{n}$ and
$\gamma_{n}$ \cite{MR1272122}.

\medskip There are several papers on the Laguerre-Freud equations for
different types of orthogonal polynomials including continuous
\cite{MR2016752}, \cite{MR2910834}, \cite{MR2556661}, discrete
\cite{MR2360113}, \cite{MR1737084}, \cite{MR842801}, \cite{MR2063533},
$D_{\omega}$ polynomials \cite{MR1662690}, \cite{MR2457103}, Laguerre-Hahn
\cite{MR3507140}, and $q$-polynomials \cite{MR3413410}.

Most of the known examples belong to the set of semiclassical orthogonal
polynomials \cite{MR1270222}, where the linear functional satisfies an
equation of the form%
\[
L\left[  \phi U\left(  \pi\right)  \right]  =L\left[  \lambda\pi\right]
,\quad\pi\in\mathbb{C}\left[  x\right]  ,
\]
called the \emph{Pearson equation }\cite{Pearson}, where $U:\mathbb{C}\left[
x\right]  \rightarrow\mathbb{C}\left[  x\right]  $ is a linear operator and
$\phi\left(  x\right)  ,$ $\lambda\left(  x\right)  $ are fixed polynomials.
The \emph{class} of the semiclassical orthogonal polynomials is defined by%
\[
c=\max\left\{  \deg\left(  \phi\right)  -2,\ \deg\left(  \phi-\lambda\right)
-1\right\}  .
\]

In this paper, we focus our attention on linear functionals defined by%
\begin{equation}
L\left[  f\right]  =%
{\displaystyle\sum\limits_{x=0}^{\infty}}
f(x)\rho\left(  x\right)  , \label{L}%
\end{equation}
where the weight function $\rho\left(  x\right)  $ is of the form%
\begin{equation}
\rho\left(  x\right)  =\frac{\left(  a_{1}\right)  _{x}\left(  a_{2}\right)
_{x}\cdots\left(  a_{p}\right)  _{x}}{\left(  b_{1}+1\right)  _{x}\left(
b_{2}+1\right)  _{x}\cdots\left(  b_{q}+1\right)  _{x}}\frac{z^{x}}{x!},
\label{weight}%
\end{equation}
and $\left(  a\right)  _{x}$ denotes the Pochhammer symbol (also called
shifted or rising factorial) defined by \cite[5.2.4]{MR2723248}
\begin{align*}
\left(  a\right)  _{0}  &  =1\\
\left(  a\right)  _{x}  &  =a\left(  a+1\right)  \cdots\left(  a+x-1\right)
,\quad x\in\mathbb{N},
\end{align*}
or by%
\[
\left(  a\right)  _{x}=\frac{\Gamma\left(  a+x\right)  }{\Gamma\left(
a\right)  },
\]
where $\Gamma\left(  z\right)  $ is the Gamma function. Note that we have%
\begin{equation}
\frac{\rho\left(  x+1\right)  }{\rho\left(  x\right)  }=\frac{\lambda\left(
x\right)  }{\phi\left(  x+1\right)  }, \label{ratio}%
\end{equation}
with%
\begin{align}
\lambda\left(  x\right)   &  =z\left(  x+a_{1}\right)  \left(  x+a_{2}\right)
\cdots\left(  x+a_{p}\right)  ,\label{phi-eta}\\
\phi\left(  x\right)   &  =x\left(  x+b_{1}\right)  \left(  x+b_{2}\right)
\cdots\left(  x+b_{q}\right)  .\nonumber
\end{align}
Hence, the weight function $\rho\left(  x\right)  $ satisfies an alternative
form of the Pearson equation\emph{ }
\begin{equation}
\Delta_{x}\left(  \phi\rho\right)  =\left(  \lambda-\phi\right)  \rho,
\label{Pearson}%
\end{equation}
where%
\[
\Delta_{x}f(x)=f(x+1)-f(x)
\]
is the forward difference operator. Using (\ref{ratio}) in (\ref{L}), we get
the Pearson equation%

\begin{equation}
L\left[  \lambda\left(  x\right)  \pi\left(  x\right)  \right]  =L\left[
\phi\left(  x\right)  \pi\left(  x-1\right)  \right]  ,\quad\pi\in
\mathbb{C}\left[  x\right]  . \label{PearsonL}%
\end{equation}

The rest of the paper is organized as follows: in Section 2 we use
(\ref{PearsonL}) and obtain two difference equations satisfied by the discrete
semiclassical orthogonal polynomials. As an example, we apply the method to
obtain the recurrence coefficients of the Meixner polynomials.

In Section 3, we derive the Laguerre-Freud equations for the Generalized Hahn
polynomials of type I, introduced in \cite{MR3227440} as part of the
classification of discrete semiclassical orthogonal polynomials of class one.
Specializing one of the parameters in the polynomials, we obtain the
recurrence coefficients of the Hahn polynomials.

We finish the paper with some remarks and future directions.

\section{Laguerre-Feud equations}

As Maroni remarks at the beginning of \cite{MR1711161}, \textquotedblleft the
history of finite-type relations is as old as the history of orthogonality
since%
\[
r(x)P_{n}(x)=%
{\displaystyle\sum\limits_{k=n-t}^{n+t}}
\lambda_{n,k}P_{k}(x),
\]
when $P_{n}(x)$ is a sequence of orthogonal polynomials and $r(x)$ is a
polynomial with $\deg\left(  r\right)  =t.$\textquotedblright\ The three-term
recurrence relation (\ref{3-term}) is the most used example, with $r(x)=x.$

We now derive difference equations for orthogonal polynomials whose linear
functional satisfies (\ref{PearsonL}). We follow an approach similar to the
one used in \cite{MR2957309} to find the Laguerre-Freud equations for the
generalized Charlier polynomials. Another method used in many articles is to
use ladder operators \cite{MR2745405}.

\begin{proposition}
Let $\left\{  P_{n}(x)\right\}  $ be a family of orthogonal polynomials with
respect to a linear functional satisfying (\ref{PearsonL}). Then, we have%
\begin{equation}
\lambda\left(  x\right)  P_{n}\left(  x+1\right)  =%
{\displaystyle\sum\limits_{k=-q-1}^{p}}
A_{k}\left(  n\right)  P_{n+k}\left(  x\right)  \label{DE1}%
\end{equation}
and%
\begin{equation}
\phi\left(  x\right)  P_{n}\left(  x-1\right)  =%
{\displaystyle\sum\limits_{k=-p}^{q+1}}
B_{k}\left(  n\right)  P_{n+k}\left(  x\right)  , \label{DE2}%
\end{equation}
for some coefficients $A_{k}\left(  n\right)  ,$ $B_{k}\left(  n\right)  .$
\end{proposition}

\begin{proof}
Since $\deg\lambda\left(  x\right)  P_{n}\left(  x+1\right)  =n+p,$ we can
write%
\[
\lambda\left(  x\right)  P_{n}\left(  x+1\right)  =%
{\displaystyle\sum\limits_{k=-n}^{p}}
A_{k}\left(  n\right)  P_{n+k}\left(  x\right)  .
\]
Using orthogonality and (\ref{PearsonL}), we have%
\begin{align*}
h_{n+k}A_{k}\left(  n\right)   &  =L\left[  \lambda\left(  x\right)
P_{n}\left(  x+1\right)  P_{n+k}\left(  x\right)  \right] \\
&  =L\left[  \phi\left(  x\right)  P_{n}\left(  x\right)  P_{n+k}\left(
x-1\right)  \right]  =0,\quad k<-q-1.
\end{align*}

Similarly, writing%
\[
\phi\left(  x\right)  P_{n}\left(  x-1\right)  =%
{\displaystyle\sum\limits_{k=-n}^{q+1}}
B_{k}\left(  n\right)  P_{n+k}\left(  x\right)  ,
\]
we get%
\begin{align*}
h_{n+k}B_{k}\left(  n\right)   &  =L\left[  \phi\left(  x\right)  P_{n}\left(
x-1\right)  P_{n+k}\left(  x\right)  \right] \\
&  =L\left[  \lambda\left(  x\right)  P_{n}\left(  x\right)  P_{n+k}\left(
x+1\right)  \right]  =0,\quad k<-p.
\end{align*}

\end{proof}

The coefficients $A_{k}\left(  n\right)  $ and $B_{k}\left(  n\right)  $ are
not independent of each other.

\begin{corollary}%
\begin{equation}
A_{k}\left(  n\right)  =\frac{h_{n}}{h_{n+k}}B_{-k}\left(  n+k\right)
,\quad-q-1\leq k\leq p. \label{AB}%
\end{equation}

\end{corollary}

\begin{proof}
If $-q-1\leq k\leq p,$ then
\begin{align*}
A_{k}\left(  n\right)   &  =\frac{1}{h_{n+k}}L\left[  \phi\left(  x\right)
P_{n}\left(  x\right)  P_{n+k}\left(  x-1\right)  \right] \\
&  =\frac{1}{h_{n+k}}L\left[  P_{n}\left(  x\right)
{\displaystyle\sum\limits_{j=-p}^{q+1}}
B_{j}\left(  n+k\right)  P_{n+k+j}\left(  x\right)  \right] \\
&  =\frac{1}{h_{n+k}}%
{\displaystyle\sum\limits_{j=-p}^{q+1}}
B_{j}\left(  n+k\right)  L\left[  P_{n}\left(  x\right)  P_{n+k+j}\left(
x\right)  \right] \\
&  =\frac{h_{n}}{h_{n+k}}B_{-k}\left(  n+k\right)  .
\end{align*}

\end{proof}

We can now state our main result.

\begin{theorem}
\label{Th1}For $-q-1\leq k\leq p,$ we have
\begin{align}
&  \gamma_{n+k+1}A_{k+1}\left(  n\right)  -\gamma_{n}A_{k+1}\left(
n-1\right)  +A_{k-1}\left(  n\right)  -A_{k-1}\left(  n+1\right) \label{req}\\
&  =\left(  \beta_{n}-\beta_{n+k}-1\right)  A_{k}\left(  n\right)  ,\nonumber
\end{align}
with%
\begin{equation}
A_{p}\left(  n\right)  =z, \label{Ap}%
\end{equation}%
\begin{equation}
A_{-q-1}\left(  n\right)  =\gamma_{n}\gamma_{n-1}\cdots\gamma_{n-q},
\label{Aq}%
\end{equation}
and%
\[
A_{p+1}\left(  n\right)  =0=A_{-q-2}\left(  n\right)  .
\]

\end{theorem}

\begin{proof}
Using (\ref{3-term}), we have
\begin{gather*}
\lambda\left(  x\right)  \left(  x+1\right)  P_{n}\left(  x+1\right)
=\lambda\left(  x\right)  P_{n+1}\left(  x+1\right) \\
+\beta_{n}\lambda\left(  x\right)  P_{n}\left(  x+1\right)  +\gamma_{n}%
\lambda\left(  x\right)  P_{n-1}\left(  x+1\right)  ,
\end{gather*}
and from (\ref{DE1})
\begin{gather}
\lambda\left(  x\right)  \left(  x+1\right)  P_{n}\left(  x+1\right)  =%
{\displaystyle\sum\limits_{k=-q}^{p+1}}
A_{k-1}\left(  n+1\right)  P_{n+k}\left(  x\right) \label{eq1}\\
+%
{\displaystyle\sum\limits_{k=-q-1}^{p}}
\beta_{n}A_{k}\left(  n\right)  P_{n+k}\left(  x\right)  +%
{\displaystyle\sum\limits_{k=-q-2}^{p-1}}
\gamma_{n}A_{k+1}\left(  n-1\right)  P_{n+k}\left(  x\right)  .\nonumber
\end{gather}

On the other hand, if we multiply (\ref{DE1}) by $x,$ we get
\[
\lambda\left(  x\right)  xP_{n}\left(  x+1\right)  =%
{\displaystyle\sum\limits_{k=-q-1}^{p}}
A_{k}\left(  n\right)  xP_{n+k}\left(  x\right)  ,
\]
and using (\ref{3-term}) we obtain%
\begin{gather}
\lambda\left(  x\right)  xP_{n}\left(  x+1\right)  =%
{\displaystyle\sum\limits_{k=-q}^{p+1}}
A_{k-1}\left(  n\right)  P_{n+k}\left(  x\right) \label{eq2}\\
+%
{\displaystyle\sum\limits_{k=-q-1}^{p}}
\beta_{n+k}A_{k}\left(  n\right)  P_{n+k}\left(  x\right)  +%
{\displaystyle\sum\limits_{k=-q-2}^{p-1}}
\gamma_{n+k+1}A_{k+1}\left(  n\right)  P_{n+k}\left(  x\right)  .\nonumber
\end{gather}
Using (\ref{DE1}), (\ref{eq1}) and (\ref{eq2}) in the identity%
\[
\lambda\left(  x\right)  P_{n}\left(  x+1\right)  =\left(  x+1\right)
\lambda\left(  x\right)  P_{n}\left(  x+1\right)  -x\lambda\left(  x\right)
P_{n}\left(  x+1\right)  ,
\]
we have%
\begin{align*}%
{\displaystyle\sum\limits_{k=-q-1}^{p}}
A_{k}\left(  n\right)  P_{n+k}\left(  x\right)   &  =%
{\displaystyle\sum\limits_{k=-q}^{p+1}}
\left[  A_{k-1}\left(  n+1\right)  -A_{k-1}\left(  n\right)  \right]
P_{n+k}\left(  x\right) \\
&  +%
{\displaystyle\sum\limits_{k=-q-1}^{p}}
\left(  \beta_{n}-\beta_{n+k}\right)  A_{k}\left(  n\right)  P_{n+k}\left(
x\right) \\
&  +%
{\displaystyle\sum\limits_{k=-q-2}^{p-1}}
\left[  \gamma_{n}A_{k+1}\left(  n-1\right)  -\gamma_{n+k+1}A_{k+1}\left(
n\right)  \right]  P_{n+k}\left(  x\right)  .
\end{align*}
Since the polynomials $P_{n}\left(  x\right)  $ are linearly independent, we
get:%
\begin{equation}
k=p+1:\quad A_{p}\left(  n+1\right)  -A_{p}\left(  n\right)  =0, \label{k0}%
\end{equation}%
\begin{equation}
k=-q-2:\quad\gamma_{n}A_{-q-1}\left(  n-1\right)  -\gamma_{n-q-1}%
A_{-q-1}\left(  n\right)  =0, \label{k1}%
\end{equation}
and for $-q-1\leq k\leq p,$
\begin{align*}
\left(  1+\beta_{n+k}-\beta_{n}\right)  A_{k}\left(  n\right)   &
=A_{k-1}\left(  n+1\right)  -A_{k-1}\left(  n\right) \\
&  +\gamma_{n}A_{k+1}\left(  n-1\right)  -\gamma_{n+k+1}A_{k+1}\left(
n\right)  .
\end{align*}

Comparing leading coefficients in (\ref{DE1}) we obtain
\[
A_{p}\left(  n\right)  =z,
\]
in agreement with (\ref{k0}).

Rewriting (\ref{k1}) as
\[
\frac{A_{-q-1}\left(  n\right)  }{A_{-q-1}\left(  n-1\right)  }=\frac
{\gamma_{n}}{\gamma_{n-q-1}},
\]
we see that%
\[
\frac{A_{-q-1}\left(  n\right)  }{A_{-q-1}\left(  q+1\right)  }=\frac
{\gamma_{n}\gamma_{n-1}\cdots\gamma_{n-q}}{\gamma_{1}\gamma_{2}\cdots
\gamma_{q+1}}.
\]
From (\ref{AB}) we have%
\[
A_{-q-1}\left(  q+1\right)  =\frac{h_{q+1}}{h_{0}}B_{q+1}\left(  0\right)  .
\]
Since $\phi\left(  x\right)  P_{n}\left(  x-1\right)  $ is a monic polynomial,
(\ref{DE2}) gives%
\begin{equation}
B_{q+1}\left(  n\right)  =1, \label{Bq}%
\end{equation}
and using (\ref{gh}) we get%
\[
\frac{h_{q+1}}{h_{0}}B_{q+1}\left(  0\right)  =\gamma_{1}\gamma_{2}%
\cdots\gamma_{q+1},
\]
proving (\ref{Aq}).
\end{proof}

\subsection{Meixner polynomials}

To illustrate the use of Theorem \ref{Th1}, we consider the family of Meixner
polynomials introduced by Josef Meixner in \cite{MR1574715}. These polynomials
are orthogonal with respect to the weight function%
\[
\rho\left(  x\right)  =\left(  a\right)  _{x}\frac{z^{x}}{x!},
\]
and using (\ref{phi-eta}) we have%
\[
\lambda\left(  x\right)  =z\left(  x+a\right)  ,\quad\phi\left(  x\right)
=x,
\]
and $p=1,\quad q=0.$

From (\ref{Ap}) and (\ref{Aq}) we get%
\begin{equation}
A_{1}\left(  n\right)  =z,\quad A_{-1}\left(  n\right)  =\gamma_{n},
\label{A1Meixner}%
\end{equation}
while (\ref{req}) gives:%

\[
k=1:\quad\left(  1+\beta_{n+1}-\beta_{n}\right)  A_{1}\left(  n\right)
=A_{0}\left(  n+1\right)  -A_{0}\left(  n\right)  ,
\]

\[
k=0:\quad A_{0}\left(  n\right)  =A_{-1}\left(  n+1\right)  -A_{-1}\left(
n\right)  +\gamma_{n}A_{1}\left(  n-1\right)  -\gamma_{n+1}A_{1}\left(
n\right)  ,
\]

and%
\[
k=-1:\quad\left(  1+\beta_{n-1}-\beta_{n}\right)  A_{-1}\left(  n\right)
=\gamma_{n}A_{0}\left(  n-1\right)  -\gamma_{n}A_{0}\left(  n\right)  .
\]
Using (\ref{A1Meixner}) we obtain%
\begin{equation}
z\left(  1+\beta_{n+1}-\beta_{n}\right)  =A_{0}\left(  n+1\right)
-A_{0}\left(  n\right)  , \label{M1}%
\end{equation}%
\begin{equation}
A_{0}\left(  n\right)  =\gamma_{n+1}-\gamma_{n}+z\left(  \gamma_{n}%
-\gamma_{n+1}\right)  =\left(  1-z\right)  \left(  \gamma_{n+1}-\gamma
_{n}\right)  , \label{M2}%
\end{equation}
and%
\begin{equation}
1+\beta_{n-1}-\beta_{n}=A_{0}\left(  n-1\right)  -A_{0}\left(  n\right)  .
\label{M3}%
\end{equation}

Summing (\ref{M1}) from $n=0$ and (\ref{M3}) from $n=1,$ we get%
\begin{align*}
z\left(  \beta_{n}-\beta_{0}+n\right)   &  =A_{0}\left(  n\right)
-A_{0}\left(  0\right)  ,\\
\beta_{n}-\beta_{0}-n  &  =A_{0}\left(  n\right)  -A_{0}\left(  0\right)  .
\end{align*}
Using (\ref{M2}) and (\ref{g0}), gives%
\[
\beta_{n}-\beta_{0}-n=z\left(  \beta_{n}-\beta_{0}+n\right)  =\left(
1-z\right)  \left(  \gamma_{n+1}-\gamma_{n}-\gamma_{1}\right)  .
\]
Therefore,%
\[
\beta_{n}=\beta_{0}+\frac{1+z}{1-z}n,
\]
and%
\begin{equation}
\gamma_{n+1}-\gamma_{n}-\gamma_{1}=\frac{2nz}{\left(  1-z\right)  ^{2}}.
\label{M4}%
\end{equation}
Summing (\ref{M4}) from $n=0,$ we conclude that%
\[
\gamma_{n}=n\gamma_{1}+\frac{n\left(  n-1\right)  z}{\left(  1-z\right)  ^{2}%
}.
\]

If we use (\ref{A1Meixner}) and (\ref{M2}) in (\ref{DE1}), we get%
\begin{gather}
z\left(  x+a\right)  P_{n}\left(  x+1\right)  =\gamma_{n}P_{n-1}\left(
x\right) \label{DEM1}\\
+\left(  1-z\right)  \left(  \gamma_{n+1}-\gamma_{n}\right)  P_{n}\left(
x\right)  +zP_{n+1}\left(  x\right)  ,\nonumber
\end{gather}
and using (\ref{AB}),%
\begin{align*}
B_{1}\left(  n\right)   &  =\frac{h_{n}}{h_{n+1}}A_{-1}\left(  n+1\right)
=\frac{A_{-1}\left(  n+1\right)  }{\gamma_{n+1}}=1,\\
B_{0}\left(  n\right)   &  =A_{0}\left(  n\right)  =\left(  1-z\right)
\left(  \gamma_{n+1}-\gamma_{n}\right)  ,\\
B_{-1}\left(  n\right)   &  =\frac{h_{n}}{h_{n-1}}A_{1}\left(  n-1\right)
=\gamma_{n}z.
\end{align*}
Hence, from (\ref{DE2}) we obtain%
\begin{equation}
xP_{n}\left(  x-1\right)  =z\gamma_{n}P_{n-1}\left(  x\right)  +\left(
1-z\right)  \left(  \gamma_{n+1}-\gamma_{n}\right)  P_{n}\left(  x\right)
+P_{n+1}\left(  x\right)  . \label{DEM2}%
\end{equation}
Setting $n=0$ in (\ref{DEM1}) and (\ref{DEM2}) gives%
\begin{align*}
z\left(  x+a\right)   &  =\left(  1-z\right)  \gamma_{1}+z\left(  x-\beta
_{0}\right)  ,\\
x  &  =\left(  1-z\right)  \gamma_{1}+x-\beta_{0},
\end{align*}
from which we find%
\[
\left(  1-z\right)  \gamma_{1}=\beta_{0}=-a+\frac{1-z}{z}\gamma_{1},
\]
and therefore%
\[
\beta_{0}=\frac{az}{1-z},\quad\gamma_{1}=\frac{az}{\left(  1-z\right)  ^{2}}.
\]

Thus, we recover the well known coefficients \cite[18.22.2 ]{MR2723248}%
\begin{equation}
\beta_{n}=\frac{n+\left(  n+a\right)  z}{1-z},\quad\gamma_{n}=\frac{n\left(
n+a-1\right)  z}{\left(  1-z\right)  ^{2}}. \label{bgM}%
\end{equation}
Using the hypergeometric representation \cite[18.20.7 ]{MR2723248}%
\[
P_{n}\left(  x\right)  =\left(  a\right)  _{n}\left(  1-\frac{1}{z}\right)
^{-n}\ _{2}F_{1}\left[
\begin{array}
[c]{c}%
-n,\ -x\\
a
\end{array}
;1-\frac{1}{z}\right]  ,
\]
one can easily verify (or re-derive) (\ref{bgM}) using (for instance) the
Mathematica package HolonomicFunctions \cite{MR3624284}.

\section{Generalized Hahn polynomials of type I}

The Generalized Hahn polynomials of type I were introduced in \cite{MR3227440}%
. They are orthogonal with respect to the weight function%
\[
\rho\left(  x\right)  =\frac{\left(  a_{1}\right)  _{x}\left(  a_{2}\right)
_{x}}{\left(  b+1\right)  _{x}}\frac{z^{x}}{x!},\quad\left\vert z\right\vert
<1,\quad b\neq-1,-2,\ldots.
\]
The first moments are given by%
\begin{align}
\mu_{0}  &  =\ _{2}F_{1}\left[
\begin{array}
[c]{c}%
a_{1},\ a_{2}\\
b+1
\end{array}
;z\right]  ,\label{GHmoments}\\
\mu_{1}  &  =z\frac{a_{1}a_{2}}{b+1}\ _{2}F_{1}\left[
\begin{array}
[c]{c}%
a_{1}+1,\ a_{2}+1\\
b+2
\end{array}
;z\right]  .\nonumber
\end{align}

Since
\[
\frac{\rho\left(  x+1\right)  }{\rho\left(  x\right)  }=\frac{z\left(
x+a_{1}\right)  \left(  x+a_{2}\right)  }{\left(  x+1\right)  \left(
x+b+1\right)  },
\]
we have%
\[
\lambda\left(  x\right)  =z\left(  x+a_{1}\right)  \left(  x+a_{2}\right)
,\quad\phi\left(  x\right)  =x\left(  x+b\right)  ,
\]
and $p=2,\quad q=1.$

We can now derive the Laguerre-Freud equations for the Generalized Hahn
polynomials of type I.

\begin{theorem}
The recurrence coefficients of the Generalized Hahn polynomials of type I
satisfy the Laguerre-Freud equations%
\begin{equation}
\left(  1-z\right)  \nabla_{n}\left(  \gamma_{n+1}+\gamma_{n}\right)
=zv_{n}\nabla_{n}\left(  \beta_{n}+n\right)  -u_{n}\nabla_{n}\left(  \beta
_{n}-n\right)  , \label{LF1}%
\end{equation}%
\begin{equation}
\Delta_{n}\nabla_{n}\left[  \left(  u_{n}-zv_{n}\right)  \gamma_{n}\right]
=u_{n}\nabla_{n}\left(  \beta_{n}-n\right)  +\nabla_{n}\left(  \gamma
_{n+1}+\gamma_{n}\right)  . \label{LF2}%
\end{equation}
with initial conditions $\beta_{0}=\frac{\mu_{1}}{\mu_{0}}$ and
\begin{equation}
\gamma_{1}=\frac{\left(  a_{1}+a_{2}-b\right)  \beta_{0}+a_{1}a_{2}}%
{1-z}-\left(  \beta_{0}+a_{1}\right)  \left(  \beta_{0}+a_{2}\right)  ,
\label{GHg1}%
\end{equation}
where
\begin{align*}
u_{n}  &  =\beta_{n}+\beta_{n-1}-n+b+1,\\
v_{n}  &  =\beta_{n}+\beta_{n-1}+n-1+a_{1}+a_{2},
\end{align*}
and%
\[
\nabla_{x}f(x)=f(x)-f(x-1).
\]

\end{theorem}

\begin{proof}
From (\ref{Ap}) and (\ref{Aq}), we get%
\begin{equation}
A_{2}\left(  n\right)  =z,\quad A_{-2}\left(  n\right)  =\gamma_{n}%
\gamma_{n-1}, \label{A2GHahn}%
\end{equation}
while (\ref{req}) gives:%
\begin{equation}
k=2:\quad A_{1}\left(  n+1\right)  -A_{1}\left(  n\right)  =z\left(
1+\beta_{n+2}-\beta_{n}\right)  , \label{GH1}%
\end{equation}%
\begin{equation}%
\begin{array}
[c]{cc}%
k=1: & A_{0}\left(  n+1\right)  -A_{0}\left(  n\right)  =A_{1}\left(
n\right)  \left(  1+\beta_{n+1}-\beta_{n}\right)  +z\left(  \gamma
_{n+2}-\gamma_{n}\right)  ,\\
k=0: & A_{-1}\left(  n+1\right)  -A_{-1}\left(  n\right)  =A_{0}\left(
n\right)  +A_{1}\left(  n\right)  \gamma_{n+1}-A_{1}\left(  n-1\right)
\gamma_{n},\\
k=-1: & A_{-2}\left(  n+1\right)  -A_{-2}\left(  n\right) \\
& =A_{-1}\left(  n\right)  \left(  1+\beta_{n-1}-\beta_{n}\right)  +\gamma
_{n}\left[  A_{0}\left(  n\right)  -A_{0}\left(  n-1\right)  \right]  ,
\end{array}
\label{GHA01}%
\end{equation}
and%
\begin{equation}
k=-2:\quad A_{-2}\left(  n\right)  \left(  1+\beta_{n-2}-\beta_{n}\right)
=A_{-1}\left(  n-1\right)  \gamma_{n}-A_{-1}\left(  n\right)  \gamma_{n-1}.
\label{GHm2}%
\end{equation}

Solving (\ref{GH1}) we get%
\begin{equation}
A_{1}\left(  n\right)  =A_{1}\left(  0\right)  +z\left(  \beta_{n+1}+\beta
_{n}+n-\beta_{0}-\beta_{1}\right)  . \label{A1GH}%
\end{equation}
Setting $n=0$ in (\ref{DE1}) we have
\[
z\left(  x+a_{1}\right)  \left(  x+a_{2}\right)  =A_{0}\left(  0\right)
+A_{1}\left(  0\right)  P_{1}\left(  x\right)  +zP_{2}\left(  x\right)  ,
\]
and using (\ref{P1})-(\ref{P2}), we get%
\begin{equation}
A_{0}\left(  0\right)  =z\left[  a_{1}a_{2}+\gamma_{1}+\left(  a_{1}%
+a_{2}\right)  \beta_{0}+\beta_{0}^{2}\right]  , \label{A00GH}%
\end{equation}
and%
\begin{equation}
A_{1}\left(  0\right)  =z\left(  a_{1}+a_{2}+\beta_{0}+\beta_{1}\right)  .
\label{A10GH}%
\end{equation}
Using (\ref{A10GH}) in (\ref{A1GH}), we obtain%
\begin{equation}
A_{1}\left(  n\right)  =z\left(  \beta_{n+1}+\beta_{n}+n+a_{1}+a_{2}\right)  .
\label{GHA1}%
\end{equation}

If we use (\ref{A2GHahn}) in (\ref{GHm2}), we get%
\[
1+\beta_{n-2}-\beta_{n}=\frac{A_{-1}\left(  n-1\right)  }{\gamma_{n-1}}%
-\frac{A_{-1}\left(  n\right)  }{\gamma_{n}},
\]
and summing from $n=2$ we see that%
\begin{equation}
n-1+\beta_{0}+\beta_{1}-\beta_{n-1}-\beta_{n}=\frac{A_{-1}\left(  1\right)
}{\gamma_{1}}-\frac{A_{-1}\left(  n\right)  }{\gamma_{n}}. \label{GHAm11}%
\end{equation}
Setting $n=0$ in (\ref{DE2}), we have%
\[
x\left(  x+b\right)  =\left(  x-\beta_{1}\right)  \left(  x-\beta_{0}\right)
-\gamma_{1}+B_{1}\left(  0\right)  \left(  x-\beta_{0}\right)  +B_{0}\left(
0\right)
\]
and hence%
\begin{equation}
B_{1}\left(  0\right)  =\beta_{0}+\beta_{1}+b, \label{B10GH}%
\end{equation}%
\begin{equation}
B_{0}\left(  0\right)  =\beta_{0}^{2}+b\beta_{0}+\gamma_{1}. \label{B00GH}%
\end{equation}
Using (\ref{AB}) with $k=-1$ and (\ref{B10GH}), we obtain
\begin{equation}
A_{-1}\left(  1\right)  =\gamma_{1}B_{1}\left(  0\right)  =\gamma_{1}\left(
\beta_{0}+\beta_{1}+b\right)  . \label{GHAm12}%
\end{equation}
Combining (\ref{GHAm11}) and (\ref{GHAm12}), we conclude that
\begin{equation}
A_{-1}\left(  n\right)  =\gamma_{n}\left(  \beta_{n}+\beta_{n-1}-n+b+1\right)
. \label{GHAm1}%
\end{equation}

If we introduce the functions%
\begin{align*}
u_{n}  &  =\frac{A_{-1}\left(  n\right)  }{\gamma_{n}}=\beta_{n}+\beta
_{n-1}-n+b+1,\\
v_{n}  &  =\frac{A_{1}\left(  n-1\right)  }{z}=\beta_{n}+\beta_{n-1}%
+n-1+a_{1}+a_{2},
\end{align*}
and use (\ref{GHA1}),(\ref{GHAm1}) in (\ref{GHA01}), we get%
\begin{align}
\nabla_{n}A_{0}  &  =zv_{n}\nabla_{n}\left(  \beta_{n}+n\right)  +z\nabla
_{n}\left(  \gamma_{n+1}+\gamma_{n}\right)  ,\nonumber\\
A_{0}  &  =\Delta_{n}\left[  \left(  u_{n}-zv_{n}\right)  \gamma_{n}\right]
,\label{GHA0}\\
\nabla_{n}A_{0}  &  =u_{n}\nabla_{n}\left(  \beta_{n}-n\right)  +\nabla
_{n}\left(  \gamma_{n+1}+\gamma_{n}\right)  .\nonumber
\end{align}
Using (\ref{AB}) with $k=0$ and (\ref{B00GH}), we obtain%
\begin{equation}
A_{0}\left(  0\right)  =B_{0}\left(  0\right)  =\beta_{0}^{2}+b\beta
_{0}+\gamma_{1}. \label{GHA002}%
\end{equation}
From (\ref{A00GH}) and (\ref{GHA002}) we have%
\begin{equation}
\left(  1-z\right)  \left[  \gamma_{1}+\left(  \beta_{0}+a_{1}\right)  \left(
\beta_{0}+a_{2}\right)  \right]  =\left(  a_{1}+a_{2}-b\right)  \beta
_{0}+a_{1}a_{2}. \label{GHb0}%
\end{equation}

Finally, if we eliminate $A_{0}$ from (\ref{GHA0}), we conclude that%
\[
zv_{n}\nabla_{n}\left(  \beta_{n}+n\right)  +z\nabla_{n}\left(  \gamma
_{n+1}+\gamma_{n}\right)  =u_{n}\nabla_{n}\left(  \beta_{n}-n\right)
+\nabla_{n}\left(  \gamma_{n+1}+\gamma_{n}\right)
\]
and%
\begin{align*}
&  \Delta_{n}\left[  \left(  u_{n}-zv_{n}\right)  \gamma_{n}\right]
-\Delta_{n}\left[  \left(  u_{n-1}-zv_{n-1}\right)  \gamma_{n-1}\right] \\
&  =u_{n}\nabla_{n}\left(  \beta_{n}-n\right)  +\nabla_{n}\left(  \gamma
_{n+1}+\gamma_{n}\right)
\end{align*}
or
\[
\Delta_{n}\nabla_{n}\left[  \left(  u_{n}-zv_{n}\right)  \gamma_{n}\right]
=u_{n}\nabla_{n}\left(  \beta_{n}-n\right)  +\nabla_{n}\left(  \gamma
_{n+1}+\gamma_{n}\right)  .
\]

\end{proof}

\subsection{Hahn polynomials}

We now consider the case $z=1.$ Under the assumptions
\[
\operatorname{Re}\left(  b-a_{1}-a_{2}\right)  >0,\quad b-a_{1}-a_{2}%
\neq1,2,\ldots,
\]
the first two moments (\ref{GHmoments}) are given by \cite[15.4(ii)]%
{MR2723248}%
\begin{align*}
\mu_{0}  &  =\frac{\Gamma\left(  b+1\right)  \Gamma\left(  b+1-a_{1}%
-a_{2}\right)  }{\Gamma\left(  b+1-a_{1}\right)  \Gamma\left(  b+1-a_{2}%
\right)  },\\
\mu_{1}  &  =\frac{a_{1}a_{2}}{b+1}\frac{\Gamma\left(  b+2\right)
\Gamma\left(  b-a_{1}-a_{2}\right)  }{\Gamma\left(  b-a_{1}\right)
\Gamma\left(  b-a_{2}\right)  }.
\end{align*}
Hence,%
\begin{equation}
\beta_{0}=\frac{\mu_{1}}{\mu_{0}}=\frac{a_{1}a_{2}}{b-a_{1}-a_{2}}.
\label{beta0}%
\end{equation}
Note that we get the same result if we set $z=1$ in (\ref{GHb0}).

Taking limits in (\ref{GHg1}) as $z\rightarrow1^{-},$ we obtain%
\[
\gamma_{1}=\frac{a_{1}a_{2}\left(  b-a_{1}\right)  \left(  b-a_{2}\right)
}{\left(  b-a_{1}-a_{2}\right)  \left(  b-1-a_{1}-a_{2}\right)  }-a_{1}%
\frac{b-a_{1}}{b-a_{1}-a_{2}}a_{2}\frac{b-a_{2}}{b-a_{1}-a_{2}},
\]
or%
\begin{equation}
\gamma_{1}=\frac{a_{1}a_{2}\left(  b-a_{1}\right)  \left(  b-a_{2}\right)
}{\left(  b-a_{1}-a_{2}\right)  ^{2}\left(  b-a_{1}-a_{2}-1\right)  },
\label{G1}%
\end{equation}
where we have used the formula \cite[15.5.1 ]{MR2723248}%
\[
\frac{d}{dz}\ _{2}F_{1}\left[
\begin{array}
[c]{c}%
a,\ b\\
c
\end{array}
;z\right]  =\frac{ab}{c}\ _{2}F_{1}\left[
\begin{array}
[c]{c}%
a+1,\ b+1\\
c+1
\end{array}
;z\right]  .
\]

When $z=1,$ the Laguerre-Freud equations (\ref{LF1})-(\ref{LF2}) decouple, and
we get%

\begin{equation}
u_{n}\nabla_{n}\left(  \beta_{n}-n\right)  =v_{n}\nabla_{n}\left(  \beta
_{n}+n\right)  , \label{LF11}%
\end{equation}%
\begin{equation}
\Delta_{n}\nabla_{n}\left[  \left(  b-a_{1}-a_{2}+2-2n\right)  \gamma
_{n}\right]  -\nabla_{n}\left(  \gamma_{n+1}+\gamma_{n}\right)  =u_{n}%
\nabla_{n}\left(  \beta_{n}-n\right)  , \label{LF21}%
\end{equation}
since in this case%
\[
u_{n}-v_{n}=b-a_{1}-a_{2}+2-2n.
\]

Solving for $\beta_{n}$ in (\ref{LF11}), we have%
\begin{equation}
\beta_{n}=\frac{2n+a_{1}+a_{2}-b-4}{2n+a_{1}+a_{2}-b}\beta_{n-1}%
-\allowbreak\frac{a_{1}+a_{2}+b}{2n+a_{1}+a_{2}-b}. \label{DEbeta}%
\end{equation}
As it is well known, the general solution of the initial value problem%
\[
y_{n+1}=c_{n}y_{n}+g_{n},\quad y_{n_{0}}=y_{0},
\]
is \cite[1.2.4]{MR2128146}%
\[
y_{n}=y_{0}%
{\displaystyle\prod\limits_{j=n_{0}}^{n-1}}
c_{j}+%
{\displaystyle\sum\limits_{k=n_{0}}^{n-1}}
\left(  g_{k}%
{\displaystyle\prod\limits_{j=k+1}^{n-1}}
c_{j}\right)  .
\]
Thus, the solution of (\ref{DEbeta}) is given by $\allowbreak$%
\begin{align*}
\beta_{n}  &  =\frac{\left(  a_{1}+a_{2}-b\right)  \left(  a_{1}%
+a_{2}-b-2\right)  }{\left(  2n+a_{1}+a_{2}-b\right)  \left(  2n+a_{1}%
+a_{2}-b-2\right)  }\beta_{0}\\
&  -\frac{\left(  a_{1}+a_{2}+b\right)  \left(  a_{1}+a_{2}-b+n-1\right)
}{\left(  2n+a_{1}+a_{2}-b\right)  \left(  2n+a_{1}+a_{2}-b-2\right)  }n,
\end{align*}
where we have used the identity%
\[%
{\displaystyle\prod\limits_{k=n_{0}}^{n_{1}}}
\frac{2n+K-2}{2n+K+2}=\frac{\left(  2n_{0}+K\right)  \left(  2n_{0}%
+K-2\right)  }{\left(  2n_{1}+K\right)  \left(  2n_{1}+K+2\right)  }.
\]
If we use the initial condition (\ref{beta0}), we conclude that%
\[
\beta_{n}=\frac{(b+2-a_{1}-a_{2})a_{1}a_{2}-n\left(  a_{1}+a_{2}+b\right)
\left(  n+a_{1}+a_{2}-b-1\right)  }{\left(  2n+a_{1}+a_{2}-b\right)
(2n+a_{1}+a_{2}-b-2)}\allowbreak.
\]

Re-writing (\ref{LF21}), we have%
\begin{gather*}
\left(  b-a_{1}-a_{2}-2n-1\right)  \gamma_{n+1}-2\left(  b-a_{1}%
-a_{2}-2n+2\right)  \gamma_{n}\\
+\left(  b-a_{1}-a_{2}-2n+5\right)  \gamma_{n-1}=u_{n}\nabla_{n}\left(
\beta_{n}-n\right)  .
\end{gather*}
Summing from $n=1,$ we get%
\begin{gather*}
\left(  b-a_{1}-a_{2}-2n-1\right)  \gamma_{n+1}+\left(  a_{1}+a_{2}%
-b+2n-3\right)  \gamma_{n}\\
+\left(  a_{1}+a_{2}-b+1\right)  \gamma_{1}=-%
{\displaystyle\sum\limits_{k=0}^{n-1}}
\beta_{k}+\beta_{n}^{2}-\beta_{0}^{2}\\
+b\left(  \beta_{n}-\beta_{0}-n\right)  -n\beta_{n}+\frac{n\left(  n-1\right)
}{2}.
\end{gather*}
The solution of this difference equation$\allowbreak$ with initial condition
(\ref{G1}) is%
\begin{align*}
\gamma_{n}  &  =-n\frac{(n+a_{1}-1)(n+a_{2}-1)(n+a_{1}-b-1)}{(2n+a_{1}%
+a_{2}-b-1)(2n+a_{1}+a_{2}-b-3)}\\
&  \times\frac{(n+a_{2}-b-1)(n+a_{1}+a_{2}-b-2)}{(2n+a_{1}+a_{2}-b-2)^{2}}.
\end{align*}

We summarize the results in the following proposition.

\begin{proposition}
The recurrence coefficients of the Hahn polynomials, orthogonal with respect
to the weight function%
\[
\rho\left(  x\right)  =\frac{\left(  a_{1}\right)  _{x}\left(  a_{2}\right)
_{x}}{x!\left(  b+1\right)  _{x}},
\]
with%
\[
\operatorname{Re}\left(  b-a_{1}-a_{2}\right)  >0,\quad b-a_{1}-a_{2}%
\neq1,2,\ldots,
\]
are given by%
\begin{equation}
\beta_{n}=\frac{(b+2-a_{1}-a_{2})a_{1}a_{2}-n\left(  a_{1}+a_{2}+b\right)
\left(  n+a_{1}+a_{2}-b-1\right)  }{\left(  2n+a_{1}+a_{2}-b\right)
(2n+a_{1}+a_{2}-b-2)}\allowbreak, \label{Hbeta}%
\end{equation}
and%
\begin{align}
\gamma_{n}  &  =-n\frac{(n+a_{1}-1)(n+a_{2}-1)(n+a_{1}-b-1)}{(2n+a_{1}%
+a_{2}-b-1)(2n+a_{1}+a_{2}-b-3)}\label{Hgamma}\\
&  \times\frac{(n+a_{2}-b-1)(n+a_{1}+a_{2}-b-2)}{(2n+a_{1}+a_{2}-b-2)^{2}%
}.\nonumber
\end{align}

\end{proposition}

This family of orthogonal polynomials was introduced by Hahn in
\cite{MR0030647}. They have the hypergeometric representation \cite{MR0054094}%
\[
P_{n}\left(  x\right)  =\frac{\left(  a_{1}\right)  _{n}\left(  a_{2}\right)
_{n}}{\left(  n+a_{1}+a_{2}-b-1\right)  _{n}}\ _{3}F_{2}\left[
\begin{array}
[c]{c}%
-n,\ -x,\ n+a_{1}+a_{2}-b-1\\
a_{1},\ a_{2}%
\end{array}
;1\right]  ,
\]
from which (\ref{Hbeta}) and (\ref{Hgamma}) can be obtained using HolonomicFunctions.

As we observed in \cite{MR3492864}, the finite family of polynomials that are
usually called \textquotedblleft Hahn polynomials\textquotedblright\ in the
literature \cite[18.19]{MR2723248} correspond to the special case%
\[
a_{1}=\alpha+1,\quad a_{2}=-N,\quad b=-N-1-\beta.
\]

\section{Conclusions}

We have presented a method that allows the computation of the recurrence
coefficients of discrete orthogonal polynomials. In some cases, a closed-form
expression can be given. We plan to extend the results to include other
families of polynomials.

\newif\ifabfull\abfullfalse
\input apreambl

\end{document}